\documentclass[11pt, reqno]{amsart}
\usepackage{amsaddr}
\usepackage{latexsym,amsfonts}
\usepackage{amssymb,amsthm,upref,amscd}
\usepackage[T1]{fontenc}
\usepackage{times}
\usepackage{amscd}
\usepackage{enumerate}
\usepackage{mathrsfs}
\usepackage{amsmath}
\usepackage{color}
\usepackage{soul}
\usepackage{indentfirst}
\usepackage{comment}
\PassOptionsToPackage{reqno}{amsmath}
\usepackage{amsmath,amsfonts,amsthm,amssymb,bbm}
\usepackage{graphicx,color,dsfont}
\usepackage{enumitem}
\usepackage{fourier}

\newtheorem{thm}{Theorem}[section]

\theoremstyle{notation}

\newcommand{\R}{\mathbb{R}}

\numberwithin{equation}{section}

\newcommand{\eps}{\epsilon}

\makeatletter \@addtoreset{equation}{section} \makeatother

\usepackage[textwidth=138mm,
textheight=215mm,
left=30mm,
right=30mm,
top=25.4mm,
bottom=25.4mm,
headheight=2.17cm,
headsep=4mm,
footskip=12mm,
heightrounded,]{geometry}

\usepackage[colorlinks=true,urlcolor=blue,
citecolor=black,linkcolor=black,linktocpage,pdfpagelabels,
bookmarksnumbered,bookmarksopen]{hyperref}
\newcounter{const}
\author[{}]{Tianxiang Gou$^{\MakeLowercase a}$, Vicen\c tiu D. R\u adulescu$^{\MakeLowercase{b,c}}$}

\address[T. Gou]{%
	\centerline{$^a$School of Mathematics and Statistics, Xi’an Jiaotong University,}
	\centerline{710049 Xi’an, Shaanxi, China}}

\address[V.D. R\u adulescu]{%
	\centerline{$^b$Faculty of Applied Mathematics, AGH University of Krak\'ow,}
	\centerline{30-059 Krak\'ow, Poland}
	\centerline{$^c$Department of Mathematics, University of Craiova,}
	\centerline{200585 Craiova, Romania}}



\email{tianxiang.gou@xjtu.edu.cn}
\email{radulescu@inf.ucv.ro}
\title[Uniqueness of positive solutions]{Uniqueness of positive solutions to fractional nonlinear elliptic equations with harmonic potential}
\begin{document}

\begin{abstract}
In this paper, we establish the uniqueness of positive solutions to the following fractional nonlinear elliptic equation with harmonic potential
\begin{align*}
(-\Delta)^s u+ \left(\omega+|x|^2\right) u=|u|^{p-2}u \quad \mbox{in}\,\, \R^n,
\end{align*}
where $n \geq 1$, $0<s<1$, $\omega>-\lambda_{1,s}$, $2<p<\frac{2n}{(n-2s)^+}$, and $\lambda_{1,s}>0$ is the lowest eigenvalue of the operator $(-\Delta)^s + |x|^2$. This solves an open question raised in \cite{SS} concerning the uniqueness of solutions to the equation.
\end{abstract}

\maketitle

\thispagestyle{empty}

\section{Introduction}

In this paper, we are concerned with the uniqueness of positive solutions to the following fractional nonlinear elliptic equation with harmonic potential,
\begin{align} \label{equ}
(-\Delta)^s u+ \left(\omega+|x|^2\right) u=|u|^{p-2}u \quad \mbox{in}\,\, \R^n,
\end{align}
where $n \geq 1$, $0<s<1$, $\omega>-\lambda_{1,s}$, $2<p<2_s^*:=\frac{2n}{(n-2s)^+}$ and $\lambda_{1,s}>0$ is the lowest eigenvalue of $(-\Delta)^s + |x|^2$ given by
\begin{align} \label{eigenvalue}
\lambda_{1, s}:=\inf \left\{\int_{\R^n} |(-\Delta)^{\frac s2} u|^2 \,dx + |x|^2|u|^2 \,dx  :  u \in \Sigma_s, \int_{\R^n} |u|^2 \,dx =1 \right\},
\end{align}
$$
\Sigma_s:=H^s(\R^n) \cap L^2(\R^n, |x|^2 \, dx).
$$
Here the fractional Laplacian $(-\Delta)^s$ is defined by
$$
(-\Delta)^s u(x):=C(n, s) P.V. \int_{\R^n} \frac{u(x)-u(y)}{|x-y|^{n+2s}} \,dy, \quad C(n, s):=\frac{s2^{2s}\Gamma\left(\frac{n+2s}{2}\right)}{\pi^{\frac{n}{2}}\gamma(1-s)}. 
$$
The Sobolev space $H^s(\R^n)$ is defined by the completion of $C^{\infty}_0(\R^n)$ under the norm
$$
\|u\|:=\|u\|_2 + [u]_s,
$$
where
$$
[u]_s:=\left(\frac{C(n, s)}{2} \int_{\R^n}  \int_{\R^n} \frac{|u(x)-u(y)|^2}{|x-y|^{n+2s}} \,dx dy \right)^{\frac 12}.
$$
Equation \eqref{equ} appears in the study of standing waves to the following fractional nonlinear Schr\"odinger equation with harmonic potential,
\begin{align} \label{equt}
\textnormal{i} \partial_t \psi+(-\Delta)^s \psi+ |x|^2 \psi=|\psi|^{p-2}\psi \quad \mbox{in}\,\, \R \times \R^n,
\end{align}
where a standing wave to \eqref{equt} is a solution of the form
$$
\psi(t,x)=e^{-\textnormal{i}\omega t} u(x), \quad \omega \in \R.
$$
Equation \eqref{equt} is of particular interest in fractional quantum mechanics, where $\psi$ denotes the probability density function of a particle trapped inside a trapping potential well modeled by $|x|^2$.

Following the early work due to Laskin \cite{L1,L2}, the study of the following nonlinear elliptic equation driven by the fractional Laplacian has attracted much attention in recent decades,
\begin{align} \label{equ11}
(-\Delta)^s u+ V(x) u=|u|^{p-2}u \quad \mbox{in}\,\, \R^n,
\end{align}
where $n \geq 1$, $0<s<1$ and $2<p<2^*_s$. 
The existence of solutions to \eqref{equ11} can be straightforwardly achieved with the help of minimax arguments, see for example \cite{W}. Mathematically, it is important and fundamental to investigate the uniqueness of solutions to \eqref{equ11}. Let  $V(x)=\omega>0$. When $s=1$, by applying shooting method, the uniqueness of positive solutions to \eqref{equ11} was established by Kwong \cite{Kw}. While $0<s<1$, the consideration of the uniqueness of solutions to \eqref{equ} is challengeable, due to the fractional Laplacian operator is nonlocal and ODE techniques adapted to deal with the case $s=1$ are no longer available. It was first proved by Amick and Toland in \cite{AT} that  \eqref{equ11} has a unique positive solution for $n=1$, $s=1/2$ and $p=3$. Later, Fall and Valdinoci in \cite{FV} established the uniqueness of positive solutions to \eqref{equ11} for $s$ close to $1$. The result has been extended by Frank, Lenzmann and Silvestre \cite{FL, FLS}, who obtained the uniqueness of positive solutions with Morse index one to \eqref{equ11} for any $0<s<1$. Very recently, Fall and Weth in \cite{FW} further showed that \eqref{equ11} admits a unique positive solution.

Observe that $\Sigma_s $ is compactly embedded into $L^q(\R^n)$ for any $2 \leq q<2^*_s$. Then, by using minimax arguments in \cite{W}, one can simply derive the existence of solutions to \eqref{equ}. However, the uniqueness of solutions to \eqref{equ} is a delicate issue. When $s=1$, by applying ODE techniques due to Yanagida and Yotsutani \cite{YY}, the uniqueness of positive solutions to \eqref{equ} was established in \cite{HO1, HO2}. While $0<s<1$, the uniqueness of positive solutions to \eqref{equ} is unknown for a long time. Recently, by employing the approach developed in \cite{FL, FLS}, Gou in \cite{Gou} successfully derived that \eqref{equ} has a unique ground state. However, the uniqueness of positive solutions to \eqref{equ} is unknown so far. As an extension of the result in \cite{Gou}, the aim of the present paper is to establish the uniqueness of positive solutions to \eqref{equ}.

\begin{thm} \label{thm}
Let $n \geq 1$, $0<s<1$, $\omega>-\lambda_{1,s}$ and $2<p<2_s^*$. Then problem \eqref{equ} admits a unique positive solution in $\Sigma_s$.
\end{thm}

To prove Theorem \ref{thm}, relying on the elements attained in \cite{Gou}, it suffices to show that any positive solution to \eqref{equ} is non-degenerate. The non-degeneracy of ground states to \eqref{equ} was derived in \cite{SS}, which has been used to establish the uniqueness of ground states to \eqref{equ} in \cite{Gou}. Nevertheless, arguments proposed in \cite{SS} are not available to verify the non-degeneracy of positive solutions to \eqref{equ}. To gain the non-degeneracy of positive solutions to \eqref{equ}, we are inspired by ingredients developed in \cite{BBDG, FW}. Let $u \in \Sigma_s$ be a positive solution to \eqref{equ}. The aim is to demonstrate that $u$ is non-degenerate, i.e. $Ker [\mathcal{L}_{+, s}]=0$, where
$$
\mathcal{L}_{+, s}=(-\Delta)^s + \left(\omega+|x|^2\right) -(p-1)|u|^{p-2},
$$
see Theorem \ref{nondegeneracy}. To attain the desired conclusion, we begin with proving that any positive solution to \eqref{equ} is radially symmetric and decreasing, see Theorem \ref{symmetry}, the proof of which is based on moving planes method. However, due to the presence of the harmonic potential, the procedure carried out in the current context is different from the one conducted in \cite{FQT}, where the radial symmetry of positive solutions to \eqref{equ11} with $V(x)=\omega>0$ was revealed. Next we need to introduce the following eigenvalue problem,
\begin{align} \label{eque} 
(-\Delta)^s v +\left(\omega +|x|^2\right) v =\Lambda |u|^{p-2} v, \quad v \in \Sigma_s.
\end{align}
It is not difficult to find that \eqref{eque} possesses a sequence of eigenvalues $0<\Lambda_1<\Lambda_2 \leq \Lambda_3\leq \cdots$, because $\Sigma_s $ is compactly embedded into $L^q(\R^n)$ for any $2 \leq q<2^*_s$. Observe that $\Lambda_1$ is simple and any eigenfunction corresponding to $\Lambda_1$ is positive. And eigenfunctions corresponding to other eigenvalues are orthogonal to the eigenfunctions corresponding to $\Lambda_1$ with respect to the scalar product in $L^2(\R^n, |u|^{p-2} \,dx)$, which are indeed sign-changing. Since $u \in \Sigma_s$ is a positive solution to \eqref{equ}, then we can conclude that $\Lambda_1=1$. At this point, the key element is to check that $\Lambda_2>p-1$. If this is true, then we know that $Ker [\mathcal{L}_{+, s}]=0$. If $Ker [\mathcal{L}_{+, s}] \neq 0$, then we would get that $p-1$ is an eigenvalue to \eqref{eque}. However, this is impossible, because $\Lambda_1$ is the first eigenvalue and $\Lambda_2$ is the second one to \eqref{eque}. And there exists no eigenvalues between $\Lambda_1$ and $\Lambda_2$. To deduce that $\Lambda_2>p-1$, we shall argue by contradiction that $\Lambda_2 \leq p-1$. In order to reach a contradiction, we first need to infer that any eigenfunction corresponding to $\Lambda_2$ is non-radial, the proof of which is beneficial from the polarization arguments presented in \cite{BBDG}. Then, making use of the Picone type result established in \cite{FW}, we are able to reach a contradiction. This completes the proof.

\section{Proof of Theorem \ref{thm}}

In this section, we are going to establish Theorem \ref{thm}. To do this, we first prove the radial symmetry and decrease of any positive solution to \eqref{equ}. In what follows, we shall use the symbol $X \lesssim Y$ to denote that $X \leq C Y$ for some proper constant $C>0$.

\begin{thm} \label{symmetry}
Let $n \geq 1$, $0<s<1$, $\omega>-\lambda_{1,s}$ and $2<p<2_s^*$. Then any positive solution to \eqref{equ} is radially symmetric and decreasing.
\end{thm}
\begin{proof}
Let $u \in \Sigma_s$ be a positive solution to \eqref{equ}. In order to prove that $u$ is radially symmetric and decreasing, we shall make use of moving planes method. Define
$$
\Sigma_{\lambda}:=\left\{x\in \R^n : x_1 > \lambda\right\},
\quad
T_{\lambda}:=\left\{x\in \R^n : x_1 = \lambda\right\},
\quad
u_{\lambda}(x):=u(x_{\lambda}), 
$$
where $x_{\lambda}:=(2\lambda-x_1, x_2, \cdots, x_n)$. By proceeding bootstrap procedure, we can derive that $u \in C^{\infty}(\R^n) \cap L^{\infty}(\R^n)$ and $u(x) \to 0$ as $|x| \to \infty$. It then follows that there exists $R>0$ such that
$$
(-\Delta)^s u+ \frac {\alpha}{2} u \leq 0, \quad |x|>R, 
$$
where $\alpha>0$. This along with \cite[Lemma 4.3]{FQT} and maximum principle gives that
\begin{align} \label{decay}
0<u(x) \lesssim \frac{1}{|x|^{n+2s}}, \quad x \in \R^n.
\end{align}
Now define
\begin{align} \label{defl}
\lambda_0:=\sup_{\lambda \leq 0} A_{\lambda},  \quad A_{\lambda}:=\left\{\lambda \leq 0 : u_{\lambda}  \leq u \,\, \mbox{in} \,\, \Sigma_{\lambda}\right\}.
\end{align}
First we verify that the set $A_{\lambda}$ is not empty. For convenience, we define
$$
w_{\lambda}(x):=\left\{
\begin{aligned}
&\left(u_{\lambda}-u\right)^+(x), \quad x \in \Sigma_{\lambda},\\
&\left(u_{\lambda}-u\right)^-(x), \quad x \in \R^n \backslash \Sigma_{\lambda},
\end{aligned}
\right.
$$
where
$$
\left(u_{\lambda}-u\right)^+(x):=\max \left\{(u_{\lambda}-u)(x), 0\right\}, \quad \left(u_{\lambda}-u\right)^-(x):=\min \left\{(u_{\lambda}-u)(x), 0\right\}.
$$
Since $u \in \Sigma_s$ is a solution to \eqref{equ}, then
\begin{align} \label{s1}
\begin{split}
&\int_{\Sigma_{\lambda}} (-\Delta)^s (u_{\lambda}-u)(u_{\lambda}-u)^+ \,dx +\int_{\Sigma_{\lambda}}\omega (u_{\lambda}-u) (u_{\lambda}-u)^+ +\left(|x_{\lambda}|^2u_{\lambda}-|x|^2 u\right) (u_{\lambda}-u)^+ \,dx \\
&=\int_{\Sigma_{\lambda}} (|u_{\lambda}|^{p-2} u_{\lambda}-|u|^{p-2} u) (u_{\lambda}-u)^+ \,dx.
\end{split}
\end{align}
Observe that
\begin{align} \label{s2}
\begin{split}
\int_{\R^n} |(-\Delta)^{\frac s 2} w_{\lambda}|^2 \,dx
&=\int_{\R^n} (-\Delta)^s w_{\lambda} w_{\lambda} \,dx=\int_{\Sigma_{\lambda}}  (-\Delta)^s w_{\lambda} w_{\lambda} \,dx +\int_{\R^n \backslash \Sigma_{\lambda}}  (-\Delta)^s w_{\lambda} w_{\lambda}  \,dx \\
&=2\int_{\Sigma_{\lambda}}  (-\Delta)^s w_{\lambda} w_{\lambda}  \,dx \leq 2 \int_{\Sigma_{\lambda}} (-\Delta)^s (u_{\lambda}-u) (u_{\lambda}-u)^+ \,dx.
\end{split}
\end{align}
Since $\lambda \leq 0$, then $|x| \leq |x_{\lambda}|$ for $x \in \Sigma_{\lambda}$ and 
\begin{align} \label{s11}
\int_{\Sigma_{\lambda}}\left(|x_{\lambda}|^2u_{\lambda}-|x|^2 u\right) (u_{\lambda}-u)^+ \,dx \geq \int_{\Sigma_{\lambda}}|x|^2\left(u_{\lambda}- u\right) (u_{\lambda}-u)^+ \,dx=\int_{\Sigma_{\lambda}}|x|^2 |w_{\lambda}|^2 \,dx.
\end{align}
Moreover, using the mean value theorem and \eqref{decay}, we can derive that
\begin{align} \label{s111}
\begin{split}
\int_{\Sigma_{\lambda}} (|u_{\lambda}|^{p-2} u_{\lambda}-|u|^{p-2} u) (u_{\lambda}-u)^+ \,dx &= \int_{\Sigma_{\lambda}} \frac{|u_{\lambda}|^{p-2} u_{\lambda}-|u|^{p-2} u}{(u_{\lambda}-u)^+} |(u_{\lambda}-u)^+|^2 \,dx \\
&\lesssim \int_{\Sigma_{\lambda}} |u_{\lambda}|^{p-2} |(u_{\lambda}-u)^+|^2 \,dx \\
&\lesssim \int_{\Sigma_{\lambda}} \frac{1} {|x_{\lambda}|^{(n+2s)(p-2)}} |w_{\lambda}|^2 \,dx \\
&\leq \left(\int_{\Sigma_{\lambda}} \frac{1} {|x_{\lambda}|^{(n+2s)p}}\,dx \right)^{\frac{p-2}{p}} \left(\int_{\Sigma_{\lambda}} |w_{\lambda}|^{p} \,dx\right)^{\frac{2}{p}},
\end{split}
\end{align}
where
$$
\int_{\Sigma_{\lambda}} \frac{1} {|x_{\lambda}|^{(n+2s)p}}\,dx =\int_{\R^n \backslash \Sigma_{\lambda}} \frac{1} {|x|^{(n+2s)p}}\,dx \leq \int_{\R^n \backslash B_{|\lambda|}(0)} \frac{1} {|x|^{(n+2s)p}}\,dx = \frac{C}{|\lambda|^{(n+2s)p-n}},
$$
where $C=C(n, s, p)>0$. Consequently, we conclude from \eqref{s111} that
\begin{align} \label{s31}
\int_{\Sigma_{\lambda}} (|u_{\lambda}|^{p-2} u_{\lambda}-|u|^{p-2} u) (u_{\lambda}-u)^+ \,dx  \lesssim \frac{1}{|\lambda|^{(n+2s)p-n}}\left(\int_{\Sigma_{\lambda}} |w_{\lambda}|^{p} \,dx\right)^{\frac{2}{p}}.
\end{align}
Combining \eqref{s1}, \eqref{s2}, \eqref{s11} and \eqref{s31} and applying the embedding inequality in $\Sigma_s$, we then obtain that
\begin{align*} 
\left(\int_{\Sigma_{\lambda}} |w|^{p} \,dx\right)^{\frac{2}{p}} 
& \lesssim  \int_{\R^n} |(-\Delta)^{\frac s 2} w_{\lambda}|^2 \,dx + \int_{\R^n} \left(\omega +|x|^2\right) |w_{\lambda}|^2 \,dx  \\
& \lesssim \int_{\Sigma_{\lambda}} (|u_{\lambda}|^{p-2} u_{\lambda}-|u|^{p-2} u) (u_{\lambda}-u)^+ \,dx \lesssim \left(\frac{1}{|\lambda|^{(n+2s)p-n}}\right)^{\frac{p-2}{p}}\left(\int_{\Sigma_{\lambda}} |w_{\lambda}|^{p} \,dx\right)^{\frac{2}{p}}.
\end{align*}
Now choosing $|\lambda|>0$ large enough, we then reach a contradiction from above. This in turn infers that $w_{\lambda}=0$ in $\Sigma_{\lambda}$ for $|\lambda|>0$ large enough. Therefore, we conclude that $u_{\lambda} \leq u$ in $\Sigma_{\lambda}$ for $|\lambda|>0$ large enough. 
As a consequence, there holds that the set $A_{\lambda}$ is not empty and $\lambda_0$ is finite defined by \eqref{defl}. 

Next we are going to demonstrate that $u = u_{\lambda_0}$ in $\Sigma_{\lambda_0}$. Suppose by contradiction that $u \neq u_{\lambda_0}$ and $u_{\lambda_0} \leq u$ in $\Sigma_{\lambda_0}$. If there exists $x_0 \in \Sigma_{\lambda_0}$ such that $u(x_0)=u_{\lambda_0}(x_0)$, by applying the fact that $u$ is a solution to \eqref{equ}, then
\begin{align} \label{s22}
\begin{split}
(-\Delta)^s u_{\lambda_0}(x_0)-(-\Delta)^s u(x_0) \leq 0,
\end{split}
\end{align}
because of $|x| \leq |x_{\lambda_0}|$.
However, we find that
\begin{align} \nonumber
(-\Delta)^s u_{\lambda_0}(x_0)-(-\Delta)^s u(x_0) \label{s23}
&=C(n, s) \int_{\R^n} \frac{u(y)-u_{\lambda_0}(y)}{|x-y|^{n+2s}} \,dy \\
&=C(n, s) \left(\int_{\Sigma_{\lambda_0}} \frac{u(y)-u_{\lambda_0}(y)}{|x_0-y|^{n+2s}} \,dy+\int_{\R^n\backslash \Sigma_{\lambda_0}} \frac{u(y)-u_{\lambda_0}(y)}{|x_0-y|^{n+2s}} \,dy \right) \\ \nonumber
&=C(n, s) \int_{\Sigma_{\lambda_0}} \left(u(y)-u_{\lambda_0}(y)\right)\left(\frac{1}{|x_0-y|^{n+2s}} -\frac{1}{|x_0-y_{\lambda_0}|^{n+2s}}\right) \,dy>0.
\end{align}
Hence we reach a contradiction from \eqref{s22} and \eqref{s23}. This obviously shows that $u_{\lambda_0}<u$ in $\Sigma_{\lambda_0}$. To derive the desired conclusion, it suffices to verify that there exists $\eps>0$ small enough such that $u_{\lambda}<u$ in $\Sigma_{\lambda}$ for any $\lambda_0<\lambda<\lambda_0+\eps$. Let us denote $P=(\lambda, 0, \cdots, 0) \in \R^n$. In view of \eqref{s1}, \eqref{s2} and \eqref{s11} and the embedding inequality in $\Sigma_s$, we know that there exists $C>0$ such that
\begin{align} \label{s3}
\left(\int_{\Sigma_{\lambda}} |w|^{p} \,dx\right)^{\frac{2}{p}} \leq C \int_{\Sigma_{\lambda}} (|u_{\lambda}|^{p-2} u_{\lambda}-|u|^{p-2} u) (u_{\lambda}-u)^+ \,dx.
\end{align}
It is simple to compute that
\begin{align} \nonumber 
\int_{\Sigma_{\lambda} \cap B_R(P)} (|u_{\lambda}|^{p-2} u_{\lambda}-|u|^{p-2} u) (u_{\lambda}-u)^+ \,dx &=\int_{\Sigma_{\lambda} \cap B_R(P)} \frac{|u_{\lambda}|^{p-2} u_{\lambda}-|u|^{p-2} u} {(u_{\lambda}-u)^+} (u_{\lambda}-u)^+\,dx \\  \label{s34}
&\lesssim \int_{\Sigma_{\lambda} \cap B_R(P) \cap supp (u_{\lambda}-u)} |w_{\lambda}|^2\,dx \\ \nonumber
& \leq \left|\Sigma_{\lambda} \cap B_R(P) \cap supp (u_{\lambda}-u)\right|^{\frac{p-2}{p}} 
\left(\int_{\Sigma_{\lambda}} |w_{\lambda}|^p\,dx\right)^{\frac 2p},
\end{align}
where we used the mean value theorem and the fact that $u \in L^{\infty}(\R^n)$. Let $R>0$ and $R_0>0$ be such that $\Sigma_{\lambda} \backslash B_R(P) \subset \R^n \backslash B_{R_0}(0)$, which are determined later. Therefore, by invoking the mean value theorem and \eqref{decay} as previously, we are able to obtain that
\begin{align*}
\int_{\Sigma_{\lambda} \backslash B_R(P)} (|u_{\lambda}|^{p-2} u_{\lambda}-|u|^{p-2} u) (u_{\lambda}-u)^+ \,dx  &\lesssim \int_{\Sigma_{\lambda} \backslash B_R(P)} |u_{\lambda}|^{p-2} |w_{\lambda}|^2 \,dx \\
& \lesssim\left(\int_{\R^n \backslash B_{R_0}(0)} \frac{1} {|x_{\lambda}|^{(n+2s)p}}\,dx \right)^{\frac{p-2}{p}} \left(\int_{\Sigma_{\lambda}} |w_{\lambda}|^{p} \,dx\right)^{\frac{2}{p}} \\
& \lesssim  \left(\frac{1}{R_0 ^{(n+2s)p-n}}\right)^{\frac{p-2}{p}} \left(\int_{\Sigma_{\lambda}} |w_{\lambda}|^{p} \,dx\right)^{\frac{2}{p}}.
\end{align*}
Now we choose $R_0>0$ large enough such that
\begin{align} \label{s311}
\int_{\Sigma_{\lambda} \backslash B_R(P)} (|u_{\lambda}|^{p-2} u_{\lambda}-|u|^{p-2} u) (u_{\lambda}-u)^+ \,dx <\frac {1} {4C}\left(\int_{\Sigma_{\lambda}} |w_{\lambda}|^{p} \,dx\right)^{\frac{2}{p}}. 
\end{align}
Then we choose $R>0$ such that  $\Sigma_{\lambda} \backslash B_R(P) \subset \R^n \backslash B_{R_0}(0)$. Moreover, applying \eqref{s34}, we choose $\eps>0$ small enough such that
\begin{align} \label{s321}
\int_{\Sigma_{\lambda} \cap B_R(P)} (|u_{\lambda}|^{p-2} u_{\lambda}-|u|^{p-2} u) (u_{\lambda}-u)^+ \,dx < \frac {1} {4C} \left(\int_{\Sigma_{\lambda}} |w|^p\,dx\right)^{\frac 2p}.
\end{align}
Accordingly, from \eqref{s311} and \eqref{s321}, there holds that
$$
\int_{\Sigma_{\lambda}} (|u_{\lambda}|^{p-2} u_{\lambda}-|u|^{p-2} u) (u_{\lambda}-u)^+ \,dx < \frac{1}{2C} \left(\int_{\Sigma_{\lambda}} |w|^p\,dx\right)^{\frac 2p}.
$$
Coming back to \eqref{s3}, we then have that $w_{\lambda}=0$ in $\Sigma_{\lambda}$. This readily indicates that $u=u_{\lambda_0}$ in $\Sigma_{\lambda_0}$. Further, by \eqref{s22} and \eqref{s23}, we can get that $\lambda_0=0$. Indeed, if $\lambda_0<0$, then
$$
(-\Delta)^s u_{\lambda_0}(x_0)-(-\Delta)^s u(x_0)<0.
$$ 
This clearly contradicts \eqref{s23}. Let $0 \leq x_1<x_1'$, then $-x_1'<-x_1<0$. Define 
$$
\lambda:=\frac{-x_1-x_1'}{2}<0.
$$ 
Since $u_{\lambda} < u$ in $\Sigma_{\lambda}$, then
$$
u(-x_1', 0, \cdots, 0)=u_{\lambda}((-x_1',  0, \cdots, 0)_{\lambda}) < u((-x_1', 0, \cdots, 0)_{\lambda})=u_{\lambda}(-x_1,  0, \cdots, 0)<u(-x_1,  0, \cdots, 0).
$$
It then follows that $u(x_1',  0, \cdots, 0)<u(x_1,  0, \cdots, 0)$, due to $u_{\lambda_0}=u$ in $\Sigma_{\lambda_0}$ and $\lambda_0=0$. At this moment, repeating the arguments above in any direction, we then get that $u$ is radially symmetric and decreasing. Thus the proof is completed.
\end{proof}

\begin{thm} \label{nondegeneracy}
Let $n \geq 1$, $0<s<1$, $\omega>-\lambda_{1,s}$ and $2<p<2_s^*$. Let $u \in \Sigma_s$ be a positive solution to \eqref{equ}. Then there holds that $Ker[\mathcal{L}_{+, s}]=0$, where
$$
\mathcal{L}_{+, s}=(-\Delta)^s + \left(\omega+|x|^2\right) -(p-1)|u|^{p-2}.
$$
\end{thm}
\begin{proof}
Let $u \in \Sigma_s$ be a positive solution to \eqref{equ}. It follows from Theorem \ref{symmetry} that $u$ is radially symmetric and decreasing. First we introduce the following eigenvalue problem,
\begin{align} \label{equ1}
(-\Delta)^s v +\left(\omega +|x|^2\right) v =\Lambda |u|^{p-2} v, \quad v \in \Sigma_s.
\end{align}
It is obvious that there exists a sequence of eigenvalues $0<\Lambda_1<\Lambda_2 \leq \Lambda_3\leq \cdots $, where $\Lambda_1$ is given by
$$
\Lambda_1=\inf \left\{\int_{\R^n} |(-\Delta)^{\frac s 2} v|^2 \,dx+ \int_{\R^n} \left(\omega +|x|^2\right) |v|^2\,dx : v \in \Sigma_s, \int_{\R^n} |v|^2 |u|^{p-2} \,dx=1\right\}.
$$
It is well-known that $\Lambda_1$ is the first eigenvalue, which is simple and possesses positive eigenfunctions. 
Moreover, eigenfunctions corresponding to the eigenvalues $\Lambda_i$ for $i \geq 2$ are orthogonal to eigenfunctions corresponding to $\Lambda_1$ with respect to the scalar product in $L^2(\R^n, |u|^{p-2} \,dx)$, which indeed changes sign. Since $u \in \Sigma_s$ solves \eqref{equ} and it is positive, then it is an eigenfunction corresponding to $\Lambda_1$. This leads to $\Lambda_1 = 1$. Therefore, the second eigenvalue can be denoted by
$$
\Lambda_2=\inf \left\{\int_{\R^n} |(-\Delta)^{\frac s 2} v|^2 \,dx+ \int_{\R^n} \left(\omega +|x|^2\right) |v|^2\,dx : v \in \Sigma_s, \int_{\R^n} |v|^2 |u|^{p-2} \,dx=1, \int_{\R^n} v |u|^{p-1} \,dx=0\right\}.
$$

In the following, we shall prove that any eigenfunction corresponding to $\Lambda_2$ is not radially symmetric. Let $v \in \Sigma_s$ be an eigenfunction corresponding to $\Lambda_2$. On the contrary, we shall assume that $v$ is radially symmetric, i.e. there exists a function $\rho: [0, \infty) \to \R$ such that $\rho(|x|)=v(x)$ for $x \in \R^n$. It follows from \cite[Proposition 7]{SS} that $v$ changes sign only once. Let $r>0$ be such that $\rho(r)=0$ and $\rho(t)>0$ for $t>r$. Let $\lambda>0$ and $P_{\lambda} v$ be the polarization of $v$ with respect to the half space $H_{\lambda}$, where
\begin{align} \label{defpl}
P_{\lambda} v(x):=
\left\{
\begin{aligned}
&\max \left\{v(x), v(\sigma_{\lambda}(x))\right\}, \quad &x \in H_{\lambda},\\
&\min \left\{v(x), v(\sigma_{\lambda}(x))\right\}, \quad &x \not\in H_{\lambda},
\end{aligned}
\right.
\end{align}
$$
H_{\lambda}:=\left\{ x \in \R^n : x_1<\lambda \right\}, \quad \sigma_{\lambda}(x):=(2\lambda-x_1, x_2, \cdots, x_n).
$$
Since $v$ changes sign, then $P_{\lambda} v$ changes sign. From the definition of the polarization, we can see that 
\begin{align} \label{p1}
(P_{\lambda} v)^{\pm}=P_{\lambda} v^{\pm},
\end{align}
where the positive and negative parts of a function $\phi : \R^n \to \R$ are respectively defined by
$$
\phi^+:=\max \left\{\phi(x), 0\right\}, \quad \phi^-:=\min \left\{\phi(x), 0\right\}, \quad \phi=\phi^++\phi^-.
$$
In view of the definition of $P_{\lambda}$, we then write that
\begin{align*} 
P_{\lambda} v(x):=
\left\{
\begin{aligned}
&\frac 12 \left (v(x) + v(\sigma_{\lambda}(x))\right) + \frac 12 \left|v(x) - v(\sigma_{\lambda}(x))\right|, \quad x \in H_{\lambda},\\
&\frac 12 \left (v(x) + v(\sigma_{\lambda}(x))\right) - \frac 12 \left|v(x) - v(\sigma_{\lambda}(x))\right|, \quad x \not\in H_{\lambda}.
\end{aligned}
\right.
\end{align*}
It then follows that
\begin{align} \label{invariant}
\begin{split}
\int_{\R^N} |P_{\lambda} v^+|^2 \,dx &=\int_{H_{\lambda}} |P_{\lambda} v^+|^2 \,dx + \int_{\R^n \backslash H_{\lambda}} |P_{\lambda} v^+|^2 \,dx \\
& = \frac 12 \int_{H_{\lambda}} |v^+|^2 + |v^+(\sigma_{\lambda}\cdot)|^2 \,dx + \frac 12 \int_{H_{\lambda}} \left||v^+|^2- |v^+(\sigma_{\lambda}\cdot)|^2\right| \,dx \\
& \quad + \frac 12 \int_{\R^n \backslash H_{\lambda}} |v^+|^2 + |v^+(\sigma_{\lambda}\cdot)|^2 \,dx- \frac 12 \int_{\R^n \backslash  H_{\lambda}} \left||v^+|^2- |v^+(\sigma_{\lambda}\cdot)|^2\right| \,dx \\
&= \int_{\R^n} |v^+|^2  \,dx.
\end{split}
\end{align}
Similarly, we can conclude that
$$
\int_{\R^n} |P_{\lambda} v^-|^2 \,dx=\int_{\R^n} |v^-|^2 \,dx.
$$
Now we assert that
$$
\int_{\R^n} |x|^2 |P_{\lambda} v^+|^2 \,dx \leq \int_{\R^n}|x|^2 |v^+|^2 \,dx.
$$
Let $R>0$ and $V_R(|x|):=\min \left\{|x|^2, R^2\right\}$ for $x \in \R^n$. Then we get that the function $x \mapsto -V_R(|x|) +R^2$ is positive, decreasing and radially symmetric. By the definition of $P_{\lambda}$ given by \eqref{defpl}, then it is invariant under the polarization $P_{\lambda}$, i.e. $P_{\lambda} (-V_R +R^2)=-V_R +R^2$. It then follows from \cite[Lemma 3]{Br} that
$$
\int_{\R^d} \left(-V_R +R^2\right) |v^+|^2 \,dx \leq \int_{\R^d} P_{\lambda}(-V_R +R^2) |P_{\lambda} v^+|^2 \,dx=\int_{\R^d} \left(-V_R +R^2\right) |P_{\lambda} v^+|^2 \,dx.
$$
This along with \eqref{invariant} gives that
$$
\int_{\R^n} V_R |P_{\lambda} v^+|^2 \,dx \leq \int_{\R^n} V_R |v^+|^2 \,dx.
$$
Passing the limit as $R$ goes to infinity, we then have the desired conclusion. Similarly, we can derive that
$$
\int_{\R^n} |x|^2|P_{\lambda} v^-|^2 \,dx \leq \int_{\R^n} |x|^2 |v^-|^2 \,dx.
$$
Consequently, there holds that
$$
\int_{\R^N} |P_{\lambda} v^{\pm}|^2 \,dx =\int_{\R^N} |v^{\pm}|^2 \,dx, \quad \int_{\R^N}|x|^2 |P_{\lambda} v^{\pm}|^2 \,dx \leq \int_{\R^N} |x|^2 |v^{\pm}|^2 \,dx.
$$
This together with \eqref{p1} then leads to
\begin{align} \label{n21}
\int_{\R^n} (P_{\lambda} v)^{\pm} P_{\lambda} v \,dx=\int_{\R^n} |(P_{\lambda} v)^{\pm}|^2\,dx=\int_{\R^n} |P_{\lambda} v^{\pm}|^2 \,dx=\int_{\R^n} |v^{\pm}|^2\,dx=\int_{\R^n} v v^{\pm} \,dx,
\end{align}
\begin{align} \label{n22}
\begin{split}
\int_{\R^n} |x|^2(P_{\lambda} v)^{\pm} P_{\lambda} v \,dx &=\int_{\R^n} |x|^2 |(P_{\lambda} v)^{\pm}|^2\,dx=\int_{\R^n} |x|^2 |P_{\lambda} v^{\pm}|^2 \,dx \\
& \leq \int_{\R^n} |x|^2 |v^{\pm}|^2 \,dx=\int_{\R^n}|x|^2 v v^{\pm} \,dx.
\end{split}
\end{align}
Since $u$ is radially symmetric and decreasing, by the definition of $P_{\lambda}$ given by \eqref{defpl}, then it is invariant under the polarization $P_{\lambda}$, i.e. $P_{\lambda} u=u$. Applying \eqref{p1} and \cite[Lemma 2]{Br}, we then get that
\begin{align} \label{n23}
\begin{split}
\int_{\R^n} |u|^{p-2}(P_{\lambda}v)(P_{\lambda}v)^{\pm} \,dx
&=\int_{\R^n}  |u|^{p-2}  |(P_{\lambda} v)^{\pm}|^2\,dx=\int_{\R^n}|P_{\lambda} u|^{p-2} |P_{\lambda} v^{\pm}|^2 \,dx \\
& \geq \int_{\R^n}  |u|^{p-2} |v^{\pm}|^2\,dx=\int_{\R^n} |u|^{p-2} vv^{\pm} \,dx,
\end{split}
\end{align}
where $P_{\lambda} u>0$, due to $u>0$. In addition, making use of \cite[Lemma 2.2]{BBDG}, we know that
\begin{align} \label{n24}
\int_{\R^n} (-\Delta)^{\frac s 2} P_{\lambda} v  (-\Delta)^{\frac s 2} (P_{\lambda} v)^{\pm} \,dx \leq \int_{\R^n} (-\Delta)^{\frac s 2} v (-\Delta)^{\frac s 2}  v^{\pm} \,dx.
\end{align}
Since $v \in \Sigma_s$ solves \eqref{equ1} with $\Lambda=\Lambda_2$, then
\begin{align} \label{n2}
\int_{\R^n} (-\Delta)^{\frac s 2} v (-\Delta)^{\frac s 2}  v^{\pm} \,dx+ \int_{\R^n} \left( \omega +|x|^2\right) vv^{\pm} \,dx =\Lambda_2 \int_{\R^n} |u|^{p-2} vv^{\pm} \,dx.
\end{align}
Therefore, utilizing \eqref{n21}-\eqref{n24}, we are able to conclude from \eqref{n2} that
\begin{align*} 
\int_{\R^n} (-\Delta)^{\frac s 2} P_{\lambda} v (-\Delta)^{\frac s 2}  (P_{\lambda} v)^{\pm} \,dx+ \int_{\R^n} \left( \omega +|x|^2\right) (P_{\lambda}v)(P_{\lambda}v)^{\pm} \,dx \leq \Lambda_2 \int_{\R^n} |u|^{p-2}  (P_{\lambda}v)(P_{\lambda}v)^{\pm} \,dx.
\end{align*}
Since $u>0$, $P_{\lambda} v^+ \geq 0$ and $P_{\lambda} v^- \leq 0$ for $v^+:=\max \left\{v(x), 0\right\}$ and  $v^-:=\min \left\{v(x), 0\right\}$, by \eqref{p1}, then
$$
\int_{\R^n} (P_{\lambda} v)^{+} |u|^{p-1} \,dx =\int_{\R^n} (P_{\lambda} v^{+}) |u|^{p-1} \,dx>0,
$$
$$
\int_{\R^n} (P_{\lambda} v)^{-} |u|^{p-1} \,dx =\int_{\R^n} (P_{\lambda} v^{-}) |u|^{p-1} \,dx <0,
$$
where $v$ is sign-changing. As a consequence, we have that there exists $\alpha_1>0$ such that
$$
\int_{\R^n} \left((P_{\lambda} v)^{+} + \alpha_1(P_{\lambda} v)^- \right) |u|^{p-1} \,dx =0.
$$
Observe that
$$
\int_{\R^n} \left|(P_{\lambda} v)^{+} + \alpha_1(P_{\lambda} v)^- \right|^2 |u|^{p-2} \,dx= \int_{\R^n} \left(\left|(P_{\lambda} v)^{+} \right|^2+ \alpha^2_1 \left|(P_{\lambda} v)^- \right|^2 \right) |u|^{p-2} \,dx >0.
$$
Therefore, we know that there exists $\alpha_2>0$ such that
$$
\alpha_2^2 \int_{\R^n} \left|(P_{\lambda} v)^{+} + \alpha_1(P_{\lambda} v)^- \right|^2 |u|^{p-2} \,dx=1.
$$
This implies that $\alpha_2 \left((P_{\lambda} v)^{+} + \alpha_1(P_{\lambda} v)^-\right) \in \Sigma_s$ is a test function for the variational characterization of $\Lambda_2$. It then follows that
\begin{align*}
\Lambda_2 & \leq \alpha_2^2 \int_{\R^n} \left|(-\Delta)^{\frac s2} (P_{\lambda} v)^{+} \right|^2 \,dx +\alpha_2^2 \int_{\R^n} \left( \omega + |x|^2\right) \left|(P_{\lambda} v)^{+}\right|^2 \,dx \\
& \quad +\alpha_1^2  \alpha_2^2  \int_{\R^n} \left|(-\Delta)^{\frac s2} (P_{\lambda} v)^{-} \right|^2 \,dx +\alpha_1^2   \alpha_2^2 \int_{\R^n} \left( \omega + |x|^2\right) \left|(P_{\lambda} v)^{-}\right|^2 \,dx \\
& \quad + 2 \alpha_1 \alpha_2^2 \int_{\R^n}(-\Delta)^{\frac s2} (P_{\lambda} v)^{+} (-\Delta)^{\frac s2} (P_{\lambda} v)^{-}\,dx \\
& \leq \alpha_2^2 \Lambda_2 \int_{\R^n} |u|^{p-2} \left|(P_{\lambda}v)^{+}\right|^2 \,dx-\alpha_2^2 \int_{\R^n}(-\Delta)^{\frac s2} (P_{\lambda} v)^{+} (-\Delta)^{\frac s2} (P_{\lambda} v)^{-}\,dx \\
& \quad + \alpha_1^2 \alpha_2^2 \Lambda_2 \int_{\R^n} |u|^{p-2}\left|(P_{\lambda}v)^{-}\right|^2 \,dx-\alpha_1^2  \alpha_2^2\int_{\R^n}(-\Delta)^{\frac s2} (P_{\lambda} v)^{+} (-\Delta)^{\frac s2} (P_{\lambda} v)^{-}\,dx \\
&\quad + 2 \alpha_1 \alpha_2^2 \int_{\R^n}(-\Delta)^{\frac s2} (P_{\lambda} v)^{+} (-\Delta)^{\frac s2} (P_{\lambda} v)^{-}\,dx \\
&=\Lambda_2 \int_{\R^2} |u|^{p-2} \left|\alpha_2 \left((P_{\lambda} v)^{+} + \alpha_1(P_{\lambda} v)^- \right)\right|^2\,dx -\alpha_2^2 (1- \alpha_1)^2 \int_{\R^n}(-\Delta)^{\frac s2} (P_{\lambda} v)^{+} (-\Delta)^{\frac s2} (P_{\lambda} v)^{-}\,dx \\
&\leq \Lambda_2,
\end{align*}
where we used the fact that
$$
 \int_{\R^n}(-\Delta)^{\frac s2} (P_{\lambda} v)^{+} (-\Delta)^{\frac s2} (P_{\lambda} v)^{-}\,dx > 0.
$$
It then gives that $\alpha_1=1$. Therefore, we know that $P_{\lambda} v=(P_{\lambda}v)^{+}+(P_{\lambda}v)^{-}$ is an eigenfunction corresponding to $\Lambda_2$. Note that $\overline{x}=(r+2\lambda, 0, \cdots, 0)\not\in H_{\lambda}$ and $\sigma_{\lambda}(\overline{x})=(-r, 0, \cdots, 0)$. This then implies that
$$
P_{\lambda} v(\overline{x}) \leq v(\sigma_{\lambda}(\overline{x}))=\rho(|\sigma_{\lambda}(\overline{x})|)=0.
$$
On the other hand, since $-\overline{x} \in H_{\lambda}$ with $|\overline{x}|>r$, then
$$
P_{\lambda} v(-\overline{x}) \geq v(-\overline{x})=v(\overline{x})=\rho(|\overline{x}|)>0.
$$
Hence we can conclude that $P_{\lambda}v$ is not radial and $P_{\lambda} v\neq v$, because $v$ is assumed to be radially symmetric. Define $\psi_{\lambda}:=v-P_{\lambda} v$. Clearly, $\psi_{\lambda}  \in \Sigma_s$ is also an eigenfunction corresponding to $\Lambda_2$. Since $P_{\lambda} v \geq v$ on $H_{\lambda}$, then $\psi_{\lambda} \leq 0$ and $\psi_{\lambda} \neq 0$ on $H_{\lambda}$. Using maximum principle, we then know that $\psi_{\lambda}<0$ on $H_{\lambda}$. It then shows that $v<v \circ \sigma_{\lambda}$ on $H_{\lambda}$ for any $\lambda>0$. Thus we are able to conclude that $v$ is strictly increasing with respect to the $x_1$ direction in the half space $H_0$. We then reach a contradiction from the assumption $v$ is radial and the fact that $v$ is sign-changing and $v(x) \to 0$ as $|x| \to \infty$.

Next we are going to prove that $\Lambda_2 >p-1$. Suppose by contradiction that $\Lambda_2 \leq p-1$.
Let $v \in \Sigma_s$ be an eigenfunction corresponding to $\Lambda_2$. Define 
$$
T_{\mu}:=\left\{x \in \R^n : x \cdot \mu =0 \right\},
$$
where $\mu \in \R^n$ is a unit vector. Denote $\sigma_{\mu}$ be the reflection with respect to $T_{\mu}$. Define
$$
\psi_{\mu}:=\frac{v-v \circ \sigma_{\mu}}{2} \in \Sigma_s.
$$
Obviously, we find that $\psi_{\mu}$ enjoys the equation
\begin{align} \label{equn}
(-\Delta)^s \psi_{\mu} + \left(\omega +|x|^2\right) \psi_{\mu}=\Lambda_2 |u|^{p-2} \psi_{\mu}.
\end{align}
Define $\varphi:=-\nabla u \cdot \mu$. Then $\varphi \in \Sigma_s$ satisfies the equation
$$
(-\Delta)^s  \varphi + \left(\omega +|x|^2\right) \varphi +2 (x \cdot \mu) u=(p-1) |u|^{p-2} \varphi.
$$
Since $u$ is radially symmetric and decreasing, then $\varphi$ is odd with respect to the refection $T_{\mu}$ and $\varphi \geq 0$ in $\{x \in \R^n : x \cdot \mu >0\}$. It immediately follows from \cite[Lemma 5.3]{FW} that $\varphi>0$ on $\{x \in \R^n : x \cdot \mu > 0\}$ and $\nabla \varphi \cdot \mu >0$ in $\{x \in \R^n : x \cdot \mu =0\}$. Let $\chi : \R \to [0, 1]$ be a smooth cut-off function such that $\chi=1$ on $(-1, 1)$ and $\chi=0$ on $(-2, 2)$. Define
$$
\psi_{\mu}^R(x):=\psi_{\mu}(x) \chi\left(\frac{|x|}{R}\right), \quad R>0.
$$
Applying \cite[Lemma 2.1]{FW}, we then get that
$$
\int_{\R^n} |(-\Delta)^{\frac s 2} \psi_{\mu}^R|^2 \,dx+ \int_{\R^n} \left(\omega +|x|^2\right) |\psi_{\mu}^R|^2\,dx-(p-1) \int_{\R^n} |u|^{p-2} |\psi_{\mu}^R|^2 \,dx \geq \int_{\left\{x \cdot \mu >0\right\} \cap \left\{y \cdot \mu >0\right\}} H^{\mu}_{\psi_{\mu}^R, \varphi} \,dx,
$$
where
$$
H^{\mu}_{\psi_{\mu}^R, \varphi}(x, y):=C(N, s) \varphi(x) \varphi(y)\left(\frac{\psi_{\mu}^R(x)}{\varphi(x)}-\frac{\psi_{\mu}^R(y)}{\varphi(y)}\right)^2 \left(\frac{1}{|x-y|^{n+2s}}-\frac{1}{|\sigma_{\mu}(x)-y|^{n+2s}}\right).
$$
Taking the limit as $R \to \infty$, we then know that
$$
\int_{\R^n} |(-\Delta)^{\frac s 2} \psi_{\mu}|^2 \,dx+ \int_{\R^n} \left(\omega +|x|^2\right) |\psi_{\mu}|^2\,dx-(p-1) \int_{\R^n} |u|^{p-2} |\psi_{\mu}|^2 \,dx \geq \int_{\left\{x \cdot \mu >0\right\} \cap \left\{y \cdot \mu >0\right\}} H^{\mu}_{\psi_{\mu}, \varphi} \,dx \geq 0.
$$
Using \eqref{equn} and the assumption that $\Lambda_2 \leq p-1$, we then obtain that
$$
0 \geq \left(\Lambda_2-p+1\right)\int_{\R^n} |u|^{p-2} |\psi_{\mu}|^2 \,dx  \geq \int_{\left\{x \cdot \mu >0\right\} \cap \left\{y \cdot \mu >0\right\}} H^{\mu}_{\psi_{\mu}, \varphi} \,dx \geq 0.
$$
As a result, we have that $\psi_{\mu}=0$. Since $v$ is nonradial, then there exists a unit vector $\overline{\mu} \in \R^n$ such that $\psi_{\overline{\mu}} \neq 0$. We then reach a contradiction. This implies that $\Lambda_2>p-1$. Since $1=\Lambda_1<p-1<\Lambda_2$, then $Ker [\mathcal{L}_{s, +}]=0$ and the proof is completed.
\end{proof}

\begin{proof}[Proof of Theorem \ref{thm}]
From Theorems \ref{symmetry} and \ref{nondegeneracy}, we know that the conditions of Lemma 2.4 in \cite{Gou} hold true. Then, replacing the roles of ground states by positive solutions to \eqref{equ} and closely following the discussions in \cite{Gou}, we are able to conclude the desired conclusion. This completes the proof.
\end{proof}

\section*{Statements and Declarations}
\subsection*{Conflict of Interests}The authors declare that there are no conflict of interests.
\subsection*{Data Availability Statement}. Data sharing is not applicable to this article as no data sets were generated or analysed during the current study.
\subsection*{Acknowledgements}  
Tianxiang Gou was supported by the National Natural Science Foundation of China (No. 12101483) and the Postdoctoral Science Foundation of China (No. 2021M702620). 
The research of Vicen\c tiu D. R\u adulescu was supported by the grant ``Nonlinear Differential Systems in Applied Sciences" of the Romanian Ministry of Research, Innovation and Digitization, within PNRR-III-C9-2022-I8/22.

\end{document}